\documentclass[11pt,a4paper]{article}
\usepackage[T1]{fontenc}
\usepackage[utf8]{inputenc}
\usepackage{lmodern}
\usepackage[margin=25mm]{geometry}
\usepackage{amsmath,amssymb,amsthm,mathtools,cite}
\usepackage[colorlinks=true,linkcolor=blue,citecolor=blue,urlcolor=blue]{hyperref}
\hypersetup{pdftitle={Scalar Entropy and Source Terms in Harmonic Ricci Flow on Surfaces},pdfauthor={Xiang-Zhi Cao},pdfkeywords={harmonic Ricci flow, scalar entropy, modified scalar curvature, free boundary}}
\allowdisplaybreaks
\newtheorem{theorem}{Theorem}[section]
\newtheorem{proposition}[theorem]{Proposition}
\newtheorem{lemma}[theorem]{Lemma}
\newtheorem{corollary}[theorem]{Corollary}
\theoremstyle{definition}
\newtheorem{definition}[theorem]{Definition}
\newtheorem{example}[theorem]{Example}
\theoremstyle{remark}
\newtheorem{remark}[theorem]{Remark}
\numberwithin{equation}{section}
\newcommand{\Ric}{\operatorname{Ric}}
\newcommand{\Sc}{\operatorname{Sc}}
\newcommand{\tr}{\operatorname{tr}}
\newcommand{\diver}{\operatorname{div}}
\newcommand{\T}{\mathcal T}
\newcommand{\dv}{\,\mathrm{d}v_g}
\newcommand{\dA}{\,\mathrm{d}A_g}
\newcommand{\dd}{\mathrm{d}}
\title{Scalar Entropy and Source Terms in Harmonic Ricci Flow on Surfaces}
\author{Xiang-Zhi Cao\thanks{aaa7756kijlp@163.com}\,\,\thanks{School of Information Engineering, Nanjing Xiaozhuang University, Nanjing 211171, China.}}
\date{}
\begin{document}
\maketitle
\begin{abstract}
The modified scalar curvature $S=R-\alpha|\nabla\phi|^2$ satisfies a reaction--diffusion equation with a nonnegative source $Q$ under harmonic Ricci flow on a surface. The relative entropy of $|S|$ need not decrease for either sign of $S$. At a nonzero constant initial value its second time derivative is the spatial variance integral of $Q$, divided by $|S|$. We construct closed examples of both signs with spatially varying tension norm. For $S<0$, normalizing by the mass of $-S$ gives an exact covariance threshold for decay and a sufficient bound in terms of the variances of $\frac{Q}{-S}$ and $-S$, the scalar mass and density bounds, and the first positive eigenvalue. On a torus this variance bound gives decay even when the covariance is strictly negative. For $S>0$ we retain the source and boundary flux in the Hamilton square formula. In the fixed harmonic free-boundary disk case the source vanishes, and equality for a nonconstant map forces $S$ to be constant. The appendices give the boundary terms for the coupled Perelman functionals.
\end{abstract}

\section{Introduction}
Hamilton's surface entropy \cite{Hamilton1988} and Perelman's functionals \cite{zbMATH05050939} provide different monotone quantities for Ricci flow. For flows coupled to a map, the modified scalar curvature is
\[
S=R-\alpha|\nabla\phi|^2.
\]
List studied the scalar-map coupling \cite{List2008}; M\"uller treated general target manifolds and obtained coupled Perelman formulas \cite{MR2961788}. On a surface, his scalar evolution equation reads
\[
S_t=\Delta S+S^2+Q,\qquad
Q=2|\Sc^\circ|^2+2\alpha|\tau(\phi)|^2\ge0,
\]
where $\Sc=\Ric-\alpha\phi^*h$ and $\Sc^\circ=\Sc-\frac{S}{2}g$. The source $Q$ prevents a direct substitution of $S$ for $R$ in the Ricci-flow surface entropy formula: its logarithmic moment has no fixed sign.

The sign defect occurs on actual flows, even when the initial scalar is constant and the initial map is weakly conformal. Let $w=|S|$ on a sign-definite interval and $E=\int_M w\log\!\left(\frac{w}{\bar w}\right)\dv$. If $S(0)=s\ne0$ is spatially constant, then
\[
 E'(0)=0,\qquad E''(0)=\frac{1}{|s|}\int_M(Q_0-\bar Q_0)^2\,\dd v_{g(0)} .
\]
Theorem~\ref{thm:departure} and Example~\ref{ex:slice} give closed solutions for each sign with $E'(t)>0$ for small $t>0$. In these examples the map contributes to $Q_0$ through its tension alone. Thus the obstruction to monotonicity is not merely the absence of a sign in an integration by parts.

For $S<0$ write $\rho=-S$, $A=\int_M\rho\dv$, $V=\operatorname{Area}_g(M)$, $E_-=\int_M\rho\log\!\left(\frac{\rho}{\bar\rho}\right)\dv$, and $\dd\nu_-=\frac{\rho}{A}\,\dd v_g$. Theorem~\ref{thm:negativecov} identifies the contribution of $Q$ to $\left(\frac{E_-}{A}\right)'$ as
$-\operatorname{Cov}_{\nu_-}(\frac{Q}{\rho},\log\rho)$.
The relevant comparison is with diffusion and the variance of $\rho$, not with zero. Corollary~\ref{cor:adverse} gives the exact instantaneous threshold and an eigenvalue criterion allowing negative covariance. Example~\ref{ex:torus} gives flows on which even the stronger nonnegative-covariance condition holds. Changing the coupling in the same torus construction gives strictly negative covariance and strictly decreasing entropy (Example~\ref{ex:adverse}). Strict positivity in the former example persists for a short time under $C^2$ perturbations of the initial data, including maps with nonzero tension. None of these short-time assertions is a propagation theorem for arbitrary initial data.

For $S>0$, we use the square-completion argument of Kunikawa and Sakurai \cite{MR4785569}, who proved the Ricci-flow boundary formula. The coupled identity retains $Q$ and the normal derivative $S_n$. Lemma~\ref{lem:rigidity} characterizes equality when the map is nonconstant. This yields the free-boundary disk theorem in Section~\ref{sec:free}, including maps with branch points. Section~\ref{sec:spectral} gives a further sufficient condition based on source oscillation and the first eigenvalue. These positive-curvature criteria are imposed along a smooth solution.

Related work includes coupled-flow structure and continuation results \cite{MR3300708,MR3729736}, Harnack estimates \cite{MR3175257}, and eigenvalue and entropy formulas \cite{MR3163480}; see also the survey \cite{MR3547931}. Guo related Hamilton's surface entropy to Perelman's entropy in the Ricci-flow setting \cite{MR3037998}. Cortissoz and Murcia obtained boundary versions of Perelman's formulas for surfaces \cite{MR4003012}; related functionals are studied in \cite{MR4452202}. The boundary problems in \cite{MR3557306,MR3448425,MR4959090,MR1387799} concern Ricci flow and do not supply existence for the general coupled free-boundary problem considered here. Appendices~\ref{sec:F} and~\ref{sec:W} give the corresponding coupled variation formulas, conditional on smooth existence.

This paper revises the author's preprint \cite{CaoEarlier}. The scalar entropy formulas below retain the trace-free, tension-field and boundary-flux terms. The counterexamples concern the unrestricted flow with evolving map. They do not refute a monotonicity statement restricted to a fixed harmonic map under additional hypotheses.

\section{Scalar evolution and entropy balances}\label{sec:balances}
Throughout, $M=M^2$ is a compact connected smooth surface, with smooth boundary if nonempty, and $(N,h)$ is a fixed smooth Riemannian target. All fields and solutions are smooth up to the spatial boundary on the time interval under consideration. The coupling constant $\alpha>0$ is fixed in the main text; the appendices also allow $\alpha=0$. Boundary integrals and boundary conditions are omitted when $M$ is closed.

\begin{definition}[Harmonic Ricci flow \cite{MR2961788}]
A pair $(g(t),\phi(t))$ solves harmonic Ricci flow if
\begin{equation}\tag{RH}\label{RH}
\partial_tg=-2\Ric+2\alpha\nabla\phi\otimes\nabla\phi=-2\Sc,
\qquad \partial_t\phi=\tau(\phi).
\end{equation}
Here
\begin{equation}\label{eq:notation}
\T=\phi^*h=\nabla\phi\otimes\nabla\phi,
\quad \Sc=\Ric-\alpha\T,
\quad S=\tr_g\Sc=R-\alpha|\nabla\phi|^2.
\end{equation}
\end{definition}
In the fixed harmonic-map subclass, $\phi$ is independent of time and $\tau_{g(t)}(\phi)=0$ for every $t$. Harmonicity at the initial time alone does not define this subclass.

The Laplacian is $\Delta=\diver\nabla$, and $R$ is scalar curvature, so $\Ric=\frac{R}{2}g$ on a surface. The tension field is $\tau(\phi)=\tr_g\nabla\dd\phi$, whereas $\tau>0$ without an argument denotes backward time in the $\mathcal W$-functional. We write $\langle\nabla\phi,\nabla f\rangle=\dd\phi(\nabla f)$.

Let $n$ be the outward unit normal and $e$ a unit tangent to $\partial M$. Our convention is
\[
k_g=g(\nabla_en,e),\qquad \nabla^\top f=(ef)e,
\qquad |\nabla^\top\phi|^2=|\dd\phi(e)|_h^2.
\]
Thus a Euclidean disk has positive boundary geodesic curvature. We set $w_n=\partial_nw$ for scalar functions and $\phi_n=\dd\phi(n)$. The measures $\dd v_g$ and $\dd A_g$ denote area and boundary length, respectively. The derivative $k_g'$ is taken at a fixed boundary point in the original flow, before pullback.

Define
\begin{equation}\label{eq:Q}
\Sc^\circ=\Sc-\frac S2g,
\quad Q=2|\Sc^\circ|^2+2\alpha|\tau(\phi)|^2,
\quad V=\int_M\dv,\quad \bar w=V^{-1}\int_Mw\dv.
\end{equation}
Since $\Ric=\frac{R}{2}g$,
\[
\Sc^\circ=-\alpha\left(\T-\tfrac12|\nabla\phi|^2g\right),
\qquad 2|\Sc|^2=S^2+2|\Sc^\circ|^2.
\]
M\"uller's scalar evolution formula \cite[Theorem 4.5]{MR2961788} and the volume variation give
\begin{equation}\label{eq:scalar}
S_t=\Delta S+S^2+Q,\qquad \partial_t\dd v_g=-S\dd v_g.
\end{equation}
The scalar equation is a local identity and applies up to a smooth boundary by continuity. It imposes no boundary condition by itself.

\subsection{Positive scalar entropy}
When $S>0$, set
\begin{equation}\label{eq:entropy}
E_\partial=\int_MS\log\frac S{\bar S}\dv.
\end{equation}
Convexity implies $E_\partial\ge0$, with equality exactly when $S$ is spatially constant.

\begin{proposition}[Full entropy balance]\label{prop:balance}
Along a smooth solution of \eqref{RH} with $S>0$,
\begin{equation}\label{eq:balance}
E_\partial'=-\int_M\frac{|\nabla S|^2}{S}\dv
+\int_M(S-\bar S)^2\dv
+\int_MQ\log\frac S{\bar S}\dv
+\int_{\partial M}S_n\log\frac S{\bar S}\dA.
\end{equation}
This identity does not require convexity or a particular boundary condition for $\phi$.
\end{proposition}
\begin{proof}
Let $A=\int_MS\dv$. Equation~\eqref{eq:scalar} gives
\[
V'=-A,\quad A'=\int_MQ\dv+\int_{\partial M}S_n\dA,
\quad (A\log\bar S)'=A'(1+\log\bar S)+\frac{A^2}{V}.
\]
Differentiation and integration by parts also give
\begin{align*}
\left(\int_MS\log S\dv\right)'={}&-\int_M\frac{|\nabla S|^2}{S}\dv+\int_MS^2\dv\\
&+\int_MQ(1+\log S)\dv+\int_{\partial M}(1+\log S)S_n\dA.
\end{align*}
Subtract the preceding identities and use $\int_MS^2\dv-\frac{A^2}{V}=\int_M(S-\bar S)^2\dv$.
\end{proof}

\subsection{Negative scalar entropy}
For a general solution with $S<0$, define
\begin{equation}\label{eq:negativeentropy}
\rho=-S>0,\qquad E_- =\int_M\rho\log\frac\rho{\bar\rho}\dv,
\qquad A(t)=\int_M\rho\dv.
\end{equation}
The zero-flux condition $\rho_n=0$ used below is a condition on the scalar and does not impose $\phi_n=0$. From \eqref{eq:scalar},
\begin{equation}\label{eq:rhoevolution}
\rho_t=\Delta\rho-\rho^2-Q,\qquad \partial_t\dd v_g=\rho\dd v_g.
\end{equation}

\begin{proposition}[General negative-entropy balance]\label{prop:negativegeneral}
On any smooth interval where $S<0$, assume that $M$ is closed or $\rho_n=0$ on its boundary. Then
\begin{align}
A'&=-\int_MQ\dv,\qquad V'=A,\label{eq:negativeAV}\\
E_-'&=-\int_M\frac{|\nabla\rho|^2}{\rho}\dv
-\int_M(\rho-\bar\rho)^2\dv
-\int_MQ\log\frac\rho{\bar\rho}\dv.\label{eq:negativegeneral}
\end{align}
\end{proposition}
\begin{proof}
Equation~\eqref{eq:negativeAV} follows by integration. The entropy and mean correction satisfy
\begin{align*}
\left(\int_M\rho\log\rho\dv\right)'
&=-\int_M\frac{|\nabla\rho|^2}{\rho}\dv-\int_M\rho^2\dv
-\int_MQ(1+\log\rho)\dv,\\
(A\log\bar\rho)'&=A'(1+\log\bar\rho)-\frac{A^2}{V}.
\end{align*}
Subtract and use $\int_M\rho^2\dv-\frac{A^2}{V}=\int_M(\rho-\bar\rho)^2\dv$.
\end{proof}
\begin{remark}
Nonnegativity of $Q$ alone does not determine the sign of the last term in \eqref{eq:negativegeneral}. It also does not guarantee preservation of $S<0$. Thus neither negative-curvature preservation nor monotonicity for the unrestricted coupled flow follows merely by changing $S$ to $-S$.
\end{remark}

\subsection{Covariance notation}
For a probability measure $\nu$, write
\[
\operatorname{Cov}_{\nu}(a,b)
=\int_Mab\,\dd\nu-\left(\int_Ma\,\dd\nu\right)
\left(\int_Mb\,\dd\nu\right).
\]
We use the elementary identity
\begin{equation}\label{eq:probabilitycov}
\operatorname{Cov}_\nu(a,b)=\frac12\int_{M\times M}
[a(x)-a(y)][b(x)-b(y)]\,\dd\nu(x)\,\dd\nu(y).
\end{equation}
In particular, the covariance of two nondecreasing functions of the same scalar variable is nonnegative.

\subsection{Departure from constant modified scalar curvature}

\begin{theorem}\label{thm:departure}
Let a smooth solution start with spatially constant $S(0)=s\ne0$. For small
$t$ set $w=|S|$, $A_w=\int_M w\dv$ and $E=\int_M w\log\!\left(\frac{w}{\bar w}\right)\dv$.
Then $E(0)=E'(0)=0$ and
\begin{equation}\label{eq:departure}
 E''(0)=\frac{1}{|s|}\int_M(Q_0-\bar Q_0)^2\,\dd v_{g(0)},\qquad
 \left(\frac{E}{A_w}\right)''\!(0)=\frac{E''(0)}{A_w(0)} .
\end{equation}
If $Q_0$ is not spatially constant, the first derivatives of both entropies are
strictly positive for all sufficiently small $t>0$. The conclusion also
holds for a smooth flow up to a boundary, without a scalar flux assumption.
\end{theorem}
\begin{proof}
Continuity and compactness keep the sign fixed for small $t$. Write
$L(x,y)=x\log\!\left(\frac{x}{y}\right)-x+y$, so $E=\int_M L(w,\bar w)\dv$.
At $(c,c)$ with $c=|s|$, $L$ and its first differential vanish, while
$\dd^2 L_{(c,c)}[(a,b),(a,b)]=\frac{(a-b)^2}{c}$. Thus neither the time
derivatives of the volume element nor the second derivatives of the two
arguments contribute at $t=0$. Differentiating the quotient
$\bar w=\frac{1}{V}\int w\dv$ at constant $w_0$ gives
$\bar w'(0)=\overline{w_t(0)}$. By \eqref{eq:scalar},
$w_t(0)=\operatorname{sgn}(s)(s^2+Q_0)$. Subtraction removes $s^2$;
integrating the resulting square gives \eqref{eq:departure}.
The mass quotient formula follows since $E(0)=E'(0)=0$.
If the integral is positive, smoothness gives
$E'(t)=tE''(0)+o(t)$ and the same argument applies to $\frac{E}{A_w}$.
\end{proof}

\begin{example}[Increasing entropy on closed surfaces]\label{ex:slice}
Let $(M,g_0)$ be the flat two-torus or the unit round two-sphere. Fix
$\alpha>0$, choose a smooth $b$ on $M$ with $b^2$ nonconstant, and put
$N=M\times (\mathbb R/2\pi\mathbb Z)$. With circle coordinate $z$ give $N$
the smooth metric
\[
 h=\dd z^2+\ell^2\exp(2b(x)\sin z)g_0(x),\qquad \ell>0 .
\]
Take $\phi_0(x)=(x,0)$. Its pullback metric is $\ell^2g_0$, so
$\Sc^\circ_0=0$ and $S_0=R_{g_0}-2\alpha\ell^2$. In tangential directions
the target connection at $z=0$ agrees tangentially with that of $g_0$
because the slice metric is a constant multiple of $g_0$. Its normal
Christoffel symbols are
$\Gamma^z_{ij}=-\tfrac12\partial_z h_{ij}|_{z=0}=-\ell^2b(x)(g_0)_{ij}$.
Taking the $g_0$-trace therefore yields
\[
 \tau_{g_0}(\phi_0)=-2\ell^2 b(x)\,\partial_z,\qquad
 Q_0=8\alpha\ell^4 b(x)^2 .
\]
On the torus $S_0=-2\alpha\ell^2<0$; on the sphere choose
$\alpha\ell^2<1$, giving $S_0=2(1-\alpha\ell^2)>0$.
Both domain and target are closed, so the smooth short-time flow exists
by \cite[Section 4.2]{MR2961788}. Theorem~\ref{thm:departure} gives strict
increase of the raw and mass-normalized entropies for each sign. At time
zero the covariance is zero in both examples. Hence an instantaneous
covariance sign at the initial time cannot replace an interval hypothesis.
\end{example}

\section{Negative modified scalar curvature}\label{sec:negative}
Throughout this section $M$ is closed or $S_n=0$. We use $\rho$, $A$ and $E_-$ from \eqref{eq:negativeentropy}. The sign $S<0$ is assumed on the interval in the entropy theorem; Lemma~\ref{lem:negativepreserve} gives a sufficient condition for its preservation.

\subsection{Ordered source-to-density ratios}
\begin{theorem}[Negative scalar entropy]\label{thm:negativecov}
Let $(g(t),\phi(t))$ be a smooth solution with $S<0$, on a closed surface or with $S_n=0$. Put
\[
H_-=\frac{E_-}{A},\qquad \dd\nu_-=\frac\rho A\dd v_g,
\qquad r_-=\frac Q\rho,\qquad
\langle r_-\rangle_{\nu_-}=\int_Mr_-\,\dd\nu_-.
\]
Then
\begin{equation}\label{eq:negativecov}
H_-'=-\frac1A\left(\int_M\frac{|\nabla\rho|^2}{\rho}\dv
+\int_M(\rho-\bar\rho)^2\dv\right)
-\operatorname{Cov}_{\nu_-}(r_-,\log\rho).
\end{equation}
Suppose the covariance is nonnegative throughout the interval. This holds, in particular, if $r_-=p_t(\rho)$ for a nonnegative nondecreasing function $p_t$ on the range of $\rho$. Then both $H_-$ and $E_-$ are nonincreasing. Their derivatives vanish at a fixed time if and only if $\rho$ is spatially constant.

Let $b(t)=\max_M\rho(\cdot,t)$ and let $\lambda_1(g(t))$ be the first positive eigenvalue of $-\Delta_g$, with Neumann conditions if $\partial M\ne\varnothing$. Then
\begin{equation}\label{eq:negativeordereddecay}
E_-(t)\le\frac{A(t)}{A(0)}\frac{V(0)}{V(t)}E_-(0)
\exp\left(-\int_0^t\frac{\lambda_1(g(s))\bar\rho(s)}{b(s)}\,\dd s\right).
\end{equation}
No sign condition on $k_g$ is required.
\end{theorem}
\begin{proof}
Write $I=\int_M\frac{|\nabla\rho|^2}{\rho}\dv$ and $D=\int_M(\rho-\bar\rho)^2\dv$. The source term decomposes as
\[
\int_MQ\log\frac\rho{\bar\rho}\dv
=A\operatorname{Cov}_{\nu_-}(r_-,\log\rho)
+\langle r_-\rangle_{\nu_-}E_-.
\]
Since $\frac{A'}{A}=-\langle r_-\rangle_{\nu_-}$, Proposition~\ref{prop:negativegeneral} gives \eqref{eq:negativecov} and
\begin{equation}\label{eq:negativeorderedE}
E_-'=-I-D-A\operatorname{Cov}_{\nu_-}(r_-,\log\rho)
-\langle r_-\rangle_{\nu_-}E_-.
\end{equation}
Formula~\eqref{eq:probabilitycov} proves the sufficient condition on $p_t$. Since $Q\ge0$ and $E_-\ge0$, both derivatives are nonpositive. Equality forces $I=0$, hence spatial constancy of $\rho$; the converse follows directly.

The Poincar\'e inequality and $\log x\le x-1$ give
\[
I\ge\frac{\lambda_1}{b}D,\qquad D\ge\bar\rho E_-.
\]
Consequently $H_-'\le-(1+\frac{\lambda_1}{b})\bar\rho H_-$. Integrate this inequality using $\bar\rho=\frac{V'}{V}$, then multiply by $A(t)$.
\end{proof}

\begin{lemma}[Preservation of the negative sign]\label{lem:negativepreserve}
Suppose $S(\cdot,0)<0$ and $S_n=0$ if there is a boundary. Assume that on the set where $S<0$,
\[
Q\le c(t)(-S),
\]
where $c\ge0$ is continuous and locally bounded in time. Then $S<0$ throughout the smooth existence interval.
\end{lemma}
\begin{proof}
Until a first zero of $\rho=-S$, one has $\rho_t\ge\Delta\rho-\rho^2-c(t)\rho$. Let $m_0=\min_M\rho(\cdot,0)>0$ and $C(t)=\int_0^tc(s)\,\dd s$. Comparison with
\[
m(t)=\frac{e^{-C(t)}}{m_0^{-1}+\int_0^te^{-C(s)}\,\dd s}>0,
\qquad m'=-m^2-cm,
\]
gives $\rho\ge m$. A first zero at a finite smooth time is therefore impossible.
\end{proof}
For a source relation $Q=\rho p_t(\rho)$, the lemma applies whenever $p_t(\rho)$ admits such a bound on the region $\rho>0$. In particular, the bound must remain finite as a possible first zero of $\rho$ is approached.

\subsection{A nonzero source on the torus}
\begin{example}\label{ex:torus}
Let $M=\mathbb T^2=(\mathbb R/2\pi\mathbb Z)^2$, let $N=\mathbb R/2\pi\mathbb Z$ have metric $\dd\theta^2$, and take
\begin{equation}\label{eq:torusdata}
\phi_0(x,y)=x\pmod{2\pi},\qquad
g_0=e^{2\varepsilon\cos y}(\dd x^2+\dd y^2),\qquad
\alpha>1,\quad 0<2\varepsilon<\alpha-1.
\end{equation}
There is $T_*>0$ such that the smooth coupled flow from these data satisfies, for $0\le t<T_*$,
\[
S<0,\qquad Q>0,\qquad
\frac Q{-S}=p_t(-S),\quad p_t\text{ strictly increasing}.
\]
Both $S$ and $\frac{Q}{-S}$ are spatially nonconstant. Thus Theorem~\ref{thm:negativecov} applies with strict entropy decay, whereas the proportional-source assumption \eqref{eq:proportional} does not hold.
\end{example}
\begin{proof}
Short-time existence and uniqueness on the closed surface follow from \cite[Section 4.2]{MR2961788}. The initial data are equivariant under translations of $x$, paired with translations of the target circle, and under $x\mapsto-x$, paired with $\theta\mapsto-\theta$. They are also invariant under $y\mapsto-y$. Uniqueness preserves these symmetries. Hence, for small time,
\[
g(t)=a(y,t)\dd x^2+b(y,t)\dd y^2,\qquad \phi(x,y,t)=x,
\]
with $a,b>0$ smooth and even in $y$. Indeed, translation equivariance first gives $\phi=x+u(y,t)$; reflection in $x$ gives $2u=0$ modulo $2\pi$, and continuity from $u(\cdot,0)=0$ gives $u=0$. The same reflection removes the mixed metric coefficient. The map $x$ is harmonic for every metric of the displayed form, since
\[
\Delta_g x=\frac1{\sqrt{ab}}\partial_x\sqrt{\frac{b}{a}}=0.
\]
The pullback tensor has rank one, so
\[
|\nabla\phi|^2=a^{-1},\qquad Q=\alpha^2a^{-2}>0.
\]

Put $s=\cos y$. At $t=0$, the conformal curvature formula gives
\[
R_0=2\varepsilon s e^{-2\varepsilon s},\qquad
\rho_0=e^{-2\varepsilon s}(\alpha-2\varepsilon s)>0,
\qquad r_{-,0}=\frac{\alpha^2e^{-2\varepsilon s}}{\alpha-2\varepsilon s}.
\]
Their derivatives are
\begin{align*}
\partial_s\rho_0
&=-2\varepsilon e^{-2\varepsilon s}(\alpha+1-2\varepsilon s)<0,\\
\partial_s r_{-,0}
&=\frac{2\varepsilon\alpha^2e^{-2\varepsilon s}
(1-\alpha+2\varepsilon s)}{(\alpha-2\varepsilon s)^2}<0.
\end{align*}
Both inequalities are uniform for $-1\le s\le1$.

The evolving scalar functions $\rho$ and $r_-$ are even and $2\pi$-periodic in $y$. They therefore define functions of $s\in[-1,1]$ whose first $s$-derivatives are continuous up to the endpoints and vary continuously with $t$. To check the endpoint assertion, write $\partial_sF=-\frac{F_y}{\sin y}$; evenness gives $F_y=0$ at $y=0,\pi$, and Taylor's formula expresses the endpoint limits in terms of $F_{yy}$. Smooth convergence as $t\downarrow0$ gives uniform convergence of these derivatives. Positivity of $\rho$ and the two strict derivative inequalities thus persist for a smaller time interval. Since $s\mapsto\rho(s,t)$ is strictly decreasing, define $p_t$ by
\[
p_t(\rho(s,t))=r_-(s,t),\qquad
p_t'=\frac{\partial_s r_-}{\partial_s\rho}>0.
\]
This proves the ratio condition. Nonconstancy of $\rho$ gives strict decay by Theorem~\ref{thm:negativecov}.
\end{proof}

\subsection{Strict covariance and nonharmonic perturbations}
The functional relation $r_-=p_t(\rho)$ is sufficient but not necessary for Theorem~\ref{thm:negativecov}. The next proposition shows that strict positivity of the covariance does not require the symmetries used in Example~\ref{ex:torus}.

\begin{proposition}[An open class of initial data]\label{prop:openinitial}
Fix the torus, circle target and coupling constant of Example~\ref{ex:torus}. There is a $C^2$-open neighborhood $\mathcal U$ of the initial pair \eqref{eq:torusdata}, within the space of smooth metric-map pairs, such that each pair in $\mathcal U$ generates a smooth solution on some interval $[0,T_*)$ with
\[
S<0,\qquad Q>0,\qquad
\operatorname{Cov}_{\nu_-}(\frac{Q}{\rho},\log\rho)>0.
\]
On this interval $E_-'<0$ and $H_-'<0$. The time $T_*$ may depend on the initial pair.
\end{proposition}
\begin{proof}
At the initial pair in Example~\ref{ex:torus}, both $\rho$ and $r_-$ are strictly decreasing functions of $s=\cos y$. Thus the integrand in \eqref{eq:probabilitycov} is nonnegative everywhere and positive on a set of positive product measure, giving strict covariance positivity.

For fixed $\alpha$, the quantities $R$, $|\nabla\phi|^2$, $\tau(\phi)$ and $Q$ depend continuously on $(g,\phi)$ in the $C^2$ topology, as do their integrals. Since $\min\rho>0$ at the base pair, the logarithm, division by $\rho$, and the probability measure $\nu_-$ also vary continuously nearby. Strict positivity of $\rho$, $Q$ and the covariance therefore holds in a $C^2$-open neighborhood $\mathcal U$. Short-time existence follows from \cite[Section 4.2]{MR2961788}. For each such smooth solution, continuity at $t=0$ in $C^2$ keeps the solution in $\mathcal U$ for a sufficiently small interval. Theorem~\ref{thm:negativecov} gives strict decay: positive covariance already excludes spatially constant $\rho$.
\end{proof}

\begin{corollary}[Examples with nonzero tension]\label{cor:tensionexample}
Keep $g_0$ in \eqref{eq:torusdata} and replace the initial map by
\[
\phi_{0,\delta}(x,y)=x+\delta\sin y\pmod{2\pi}.
\]
For every sufficiently small nonzero $\delta$, the resulting coupled flow satisfies the conclusions of Proposition~\ref{prop:openinitial}, and $\tau(\phi(t))$ is not identically zero for all sufficiently small $t\ge0$.
\end{corollary}
\begin{proof}
The perturbed initial pairs converge to the base pair in $C^2$. In the unit parallel frame on the target circle, conformal invariance of the surface tension equation gives
\[
\tau_{g_0}(\phi_{0,\delta})
=-\delta e^{-2\varepsilon\cos y}\sin y\,\partial_\theta.
\]
This is nonzero at $y=\frac{\pi}{2}$, and remains nonzero there for a sufficiently small time by continuity. Apply Proposition~\ref{prop:openinitial}.
\end{proof}
The neighborhood argument proves short-time realization of the covariance condition. It gives neither a uniform existence interval for all nearby data nor preservation of the sign up to a singular time. No functional dependence of $\frac{Q}{\rho}$ on $\rho$ is assumed for the perturbed solutions.

\subsection{How much negative covariance is allowed?}
\begin{corollary}\label{cor:adverse}
Keep the hypotheses of the identity \eqref{eq:negativecov}, without any
covariance sign assumption. Write
\[
 I=\int_M\frac{|\nabla\rho|^2}{\rho}\dv,\quad
 D=\int_M(\rho-\bar\rho)^2\dv,\quad
 C=\operatorname{Cov}_{\nu_-}(r_-,\log\rho).
\]
At any fixed time the necessary and sufficient conditions are
\begin{align}
 H_-'\le0 &\quad\Longleftrightarrow\quad A C\ge-(I+D),\label{eq:thresholdH}\\
 E_-'\le0 &\quad\Longleftrightarrow\quad A C\ge-(I+D)-\langle r_-\rangle_{\nu_-}E_- .
 \label{eq:thresholdE}
\end{align}
Let $a=\min_M\rho$, $b=\max_M\rho$ and $\lambda_1$ be as in
Theorem~\ref{thm:negativecov}. A sufficient condition independent of the
sign of $C$ is
\begin{equation}\label{eq:oscweak}
 \sqrt{A b\,\operatorname{Var}_{\nu_-}(r_-)}
      \le a\left(1+\frac{\lambda_1}{b}\right)\sqrt D .
\end{equation}
Here $\operatorname{Var}_{\nu}(v)=\operatorname{Cov}_{\nu}(v,v)$.
If the condition holds along a time interval, both entropies are
nonincreasing there; strict inequality implies strict decrease at that time.
This variance bound is sufficient and is not asserted to have a best constant.
\end{corollary}
\begin{proof}
The equivalences are rearrangements of \eqref{eq:negativecov} and
\eqref{eq:negativeorderedE}. Since the mean minimizes the squared distance
to a constant and $\log$ is $a^{-1}$-Lipschitz on $[a,b]$,
\[
 \operatorname{Var}_{\nu_-}(\log\rho)
 \le\int_M(\log\rho-\log\bar\rho)^2\,\dd\nu_-
 \le\frac{bD}{Aa^2}.
\]
Cauchy--Schwarz bounds $A|C|$ by
$\sqrt{A b\,\operatorname{Var}_{\nu_-}(r_-)}\frac{\sqrt{D}}{a}$.
The Poincar\'e estimate $I\ge(\frac{\lambda_1}{b})D$ finishes the proof.
When $D=0$, both derivatives vanish directly, regardless of the source
variance. This is an assertion at one time, as \eqref{eq:departure} illustrates.
\end{proof}

\begin{example}[Decay with strictly negative covariance]\label{ex:adverse}
Use the torus, target circle, initial metric and map of
Example~\ref{ex:torus}, but fix $0<\alpha<1$. For all sufficiently small
$\varepsilon>0$ the resulting flow has $C<0$, $H_-'<0$ and $E_-'<0$
on a nonempty interval starting at $t=0$.
\end{example}
\begin{proof}
Choose $2\varepsilon<\min(\alpha,1-\alpha)$. The initial formulas in
Example~\ref{ex:torus} show that $\rho_0$ is strictly decreasing in
$s=\cos y$, whereas $r_{-,0}$ is strictly increasing. The pairwise
covariance identity therefore gives $C(0)<0$. As a function of $\rho_0$
the ratio is Lipschitz with constant
\[
 K_\varepsilon=\max_{-1\le s\le1}
 \frac{\alpha^2(1-\alpha+2\varepsilon s)}
 {(\alpha-2\varepsilon s)^2(1+\alpha-2\varepsilon s)}
 \longrightarrow\frac{1-\alpha}{1+\alpha}<1 .
\]
This follows by dividing the two $s$-derivatives; all denominators stay
uniformly away from zero. Write $a=\min\rho_0$, $b=\max\rho_0$, so $a,b\to\alpha$.
The pairwise variance formula gives
$\operatorname{Var}_{\nu_-}(r_{-,0})\le K_\varepsilon^2\operatorname{Var}_{\nu_-}(\rho_0)
\le K_\varepsilon^2 \frac{bD}{A}$. The last inequality compares $\rho_0$ to
the constant $\bar\rho_0$ in $L^2(\nu_-)$. Now $K_\varepsilon b<a$ for
small $\varepsilon$, and $D>0$. Thus \eqref{eq:oscweak} holds strictly at
$t=0$, even without its eigenvalue term. Smooth short-time existence and
continuity of $C$, $E_-'$ and $H_-'$ preserve all three strict signs for
a short time. No persistence of the Lipschitz relation is required.
\end{proof}

\subsection{Proportional sources}
\begin{corollary}[Proportional sources]\label{thm:proportional}
Let $(g(t),\phi(t))$, $0\le t<T$, be a smooth solution of \eqref{RH}, with $S(\cdot,0)<0$, and suppose $M$ is closed or $S_n=0$. Assume that a nonnegative continuous function $c(t)$, bounded on each compact time interval, satisfies
\begin{equation}\label{eq:proportional}
Q(x,t)=-c(t)S(x,t)
\end{equation}
throughout the interval. Then $S<0$ persists and
\begin{equation}\label{eq:proportionalidentity}
E_-'=-\int_M\frac{|\nabla\rho|^2}{\rho}\dv
-\int_M(\rho-\bar\rho)^2\dv-c(t)E_-\le0.
\end{equation}
At each fixed time, equality holds if and only if $\rho$ is spatially constant. Put
$A_0=\int_M\rho(\cdot,0)\,\dd v_{g(0)}$, $V_0=V(0)$ and $C(t)=\int_0^tc(s)\,\dd s$. Then
\begin{align}
A(t)&=A_0e^{-C(t)},\qquad
V(t)=V_0+A_0\int_0^te^{-C(s)}\,\dd s,\label{eq:proportionalAV}\\
E_-(t)&\le e^{-C(t)}\frac{V_0}{V(t)}E_-(0).\label{eq:proportionaldecay}
\end{align}
\end{corollary}
\begin{proof}
Lemma~\ref{lem:negativepreserve} preserves $S<0$. In \eqref{eq:negativeorderedE}, the ratio $r_-=c(t)$ has zero covariance, which gives \eqref{eq:proportionalidentity} and its equality case. Equations~\eqref{eq:negativeAV} reduce to $A'=-cA$ and $V'=A$, proving \eqref{eq:proportionalAV}. Dropping the spectral factor from \eqref{eq:negativeordereddecay} gives \eqref{eq:proportionaldecay}.
\end{proof}

Theorem~\ref{thm:negativecov} also supplies the spectral factor in \eqref{eq:negativeordereddecay} in this case.

\begin{corollary}[Weakly conformal harmonic class]\label{thm:negative}
Let $(g(t),\phi)$ be a smooth solution of \eqref{RH}. Suppose $\phi$ is time independent, harmonic and weakly conformal at every time, and $S(\cdot,0)<0$. Under the scalar boundary condition of Corollary~\ref{thm:proportional},
\begin{equation}\label{eq:negativeidentity}
E_-'=-\int_M\frac{|\nabla\rho|^2}{\rho}\dv
-\int_M(\rho-\bar\rho)^2\dv,
\end{equation}
$A=A_0$ is constant, $V(t)=V_0+A_0t$, and
\begin{equation}\label{eq:negdecay}
E_-(t)\le\frac{V_0}{V_0+A_0t}E_-(0).
\end{equation}
\end{corollary}
\begin{proof}
Weak conformality and harmonicity give $Q=0$. Apply Corollary~\ref{thm:proportional} with $c=0$.
\end{proof}

\begin{example}[Closed hyperbolic surfaces]\label{ex:hyperbolic}
Let $(M,h)$ be a closed hyperbolic surface of genus $\gamma\ge2$, with $K_h=-1$. Take $(M,h)$ as target and $\phi=\operatorname{id}_M$. For $\alpha>0$, any smooth initial $\omega_0$ satisfying $1+\alpha+\Delta_h\omega_0>0$ gives a short-time solution $g=e^{2\omega}h$ to which Corollary~\ref{thm:negative} applies. The initial data can be chosen so that $E_-'(0)<0$.
\end{example}
\begin{proof}
The identity map is harmonic and weakly conformal for every $e^{2\omega}h$, with $|\nabla\phi|^2=2e^{-2\omega}$. Hence
\[
S=-2e^{-2\omega}(1+\alpha+\Delta_h\omega),\qquad
\omega_t=e^{-2\omega}(1+\alpha+\Delta_h\omega).
\]
Short-time smooth existence and uniqueness follow from \cite[Theorem 1.1]{MantegazzaMartinazzi}. The assumed initial sign and Corollary~\ref{thm:negative} give $S<0$ on the smooth interval.

For strict decay, take a nonconstant eigenfunction $\zeta$ with $\Delta_h\zeta=-\lambda\zeta$, $\lambda>0$, and $\omega_0=\varepsilon\zeta$. For sufficiently small nonzero $\varepsilon$,
\[
\rho_0=2e^{-2\varepsilon\zeta}(1+\alpha-\varepsilon\lambda\zeta)>0.
\]
Its derivative as a function of $s=\zeta$ is
$-2\varepsilon e^{-2\varepsilon s}[\lambda+2(1+\alpha)-2\varepsilon\lambda s]$,
which has no zeros on the range of $\zeta$ for small $|\varepsilon|$. Thus $\rho_0$ is nonconstant and \eqref{eq:negativeidentity} gives strict decay. Moreover,
\[
A_0=2\int_M(1+\alpha+\Delta_h\omega)\,\dd v_h
=2(1+\alpha)\operatorname{Area}_h(M)=8\pi(1+\alpha)(\gamma-1).
\]
The map is nonconstant, so the metric coupling is nonzero, although $Q=0$.
\end{proof}

\section{Positive modified scalar curvature}\label{sec:positive}
Throughout this section $S>0$.
Let $\psi$ be the unique zero-mean solution of
\begin{equation}\label{eq:poisson}
\Delta\psi=\bar S-S,\qquad\psi_n=0,\qquad\int_M\psi\dv=0.
\end{equation}
On a closed surface only the equation and normalization are imposed. Define
\begin{equation}\label{eq:dissipation}
\begin{aligned}
\mathcal D={}&\int_MS|\nabla\psi-\nabla\log S|^2\dv
+2\int_M\left|\nabla^2\psi-\frac{\Delta\psi}{2}g\right|^2\dv\\
&+\alpha\int_M|\nabla\phi|^2|\nabla\psi|^2\dv
+2\int_{\partial M}k_g|\nabla^\top\psi|^2\dA.
\end{aligned}
\end{equation}
The boundary term can be negative when $k_g$ changes sign; nonnegativity then requires a separate argument.

\begin{proposition}[Square formulation]\label{prop:square}
Under the hypotheses of Proposition~\ref{prop:balance},
\begin{equation}\label{eq:square}
E_\partial'=-\mathcal D+\int_MQ\log\frac S{\bar S}\dv
+\int_{\partial M}S_n\log\frac S{\bar S}\dA.
\end{equation}
\end{proposition}
\begin{proof}
We follow the square-completion argument of Kunikawa and Sakurai \cite{MR4785569}, retaining the additional coupling and source terms. Set $D=\int_M(S-\bar S)^2\dv=\int_M(\Delta\psi)^2\dv$.
Reilly's formula \cite{Reilly}, with our outward-normal convention and $\psi_n=0$, gives
\[
D=2\int_M\left|\nabla^2\psi-\frac{\Delta\psi}{2}g\right|^2\dv
+\int_MR|\nabla\psi|^2\dv+2\int_{\partial M}k_g|\nabla^\top\psi|^2\dA.
\]
Moreover $\int_M\langle\nabla\psi,\nabla S\rangle\dv=D$. Expanding the first square in \eqref{eq:dissipation} and using $S-R=-\alpha|\nabla\phi|^2$ yields
$-\mathcal D=D-\int_M\frac{|\nabla S|^2}{S}\dv$.
Equation~\eqref{eq:balance} proves the assertion.
\end{proof}

\begin{lemma}[Equality in the square formula]\label{lem:rigidity}
Suppose $S>0$, $\alpha>0$, and either $M$ is closed or $k_g\ge0$. If $\phi$ is nonconstant, then $\mathcal D=0$ if and only if $S$ is spatially constant.
\end{lemma}
\begin{proof}
If $S$ is constant, the normalized Poisson solution is $\psi=0$. Conversely, $\mathcal D=0$ implies
\[
\nabla S=S\nabla\psi,\qquad
\nabla^2\psi=\tfrac12(\bar S-S)g,\qquad
|\nabla\phi|^2|\nabla\psi|^2=0.
\]
As $\phi$ is nonconstant, $|\nabla\phi|>0$ on a nonempty interior open set. There $\nabla\psi=0$, and the Hessian identity gives $S=\bar S$. Fix a point in this set and write $Z=\nabla\psi$. Along any smooth path $\gamma$ starting there,
\[
\nabla_{\dot\gamma}Z=\tfrac12(\bar S-S)\dot\gamma,
\qquad \frac{\dd S}{\dd s}=S\langle Z,\dot\gamma\rangle.
\]
Uniqueness for this ordinary differential system with initial data $(Z,S)=(0,\bar S)$ gives $Z=0$ and $S=\bar S$ along the path. Connectedness and continuity up to the boundary finish the proof.
\end{proof}

\subsection{A source-structure criterion}
\begin{theorem}[Sources nonincreasing in $S$]\label{thm:source}
Suppose $S>0$ along a smooth solution of \eqref{RH}. Assume either that $M$ is closed, or that $S_n=0$ and $k_g\ge0$ throughout the interval. Suppose that, at each time, there is a nonnegative nonincreasing function $q_t$ on the range of $S(\cdot,t)$ such that
\begin{equation}\label{eq:sourcehyp}
Q(x,t)=q_t(S(x,t)).
\end{equation}
Then $E_\partial'\le-\mathcal D\le0$. More precisely, with $b=\max_M S$,
\begin{equation}\label{eq:sourcedecay}
E_\partial'\le-\mathcal D-\frac{\bar Q}{2b^2}\int_M(S-\bar S)^2\dv.
\end{equation}
In particular, $E_\partial'<0$ at any time when $S$ is nonconstant and $\bar Q>0$.
\end{theorem}
\begin{proof}
The covariance identity gives
\begin{align*}
&\int_M(Q-\bar Q)(\log S-\overline{\log S})\dv\\
&\quad=\frac1{2V}\int_{M\times M}
[q_t(S(x))-q_t(S(y))][\log S(x)-\log S(y)]
\,\dd v_g(x)\,\dd v_g(y)\le0.
\end{align*}
Since $(-\log s)''\ge b^{-2}$ for $0<s\le b$, Taylor's inequality at $\bar S$ yields
\begin{equation}\label{eq:Jensen}
\int_M\log\frac S{\bar S}\dv\le-\frac1{2b^2}\int_M(S-\bar S)^2\dv.
\end{equation}
Decompose the source integral into its covariance and mean parts, and use $\bar Q\ge0$ in \eqref{eq:square}.
\end{proof}
\begin{remark}
Condition~\eqref{eq:sourcehyp} is imposed along the solution. It does not imply that such a relation is preserved from initial data. It permits nonzero source terms algebraically, but this alone does not construct a coupled solution realizing them. Also, when $M$ is orientable with nonempty boundary, $S>0$ and $k_g\ge0$ imply $R>0$ and hence $\chi(M)>0$ by Gauss--Bonnet; thus the theorem does not remove the disk-topology restriction in that setting.
\end{remark}

\subsection{Division by the total scalar mass}
\begin{proposition}\label{prop:positivecov}
Suppose $S>0$ and $M$ is closed or $S_n=0$. Set
\[
A_+=\int_MS\dv,\qquad H_+=\frac{E_\partial}{A_+},
\qquad \dd\nu_+=\frac{S}{A_+}\dd v_g,\qquad r_+=\frac QS.
\]
Then
\begin{equation}\label{eq:positivecov}
H_+'=-\frac{\mathcal D}{A_+}
+\operatorname{Cov}_{\nu_+}(r_+,\log S).
\end{equation}
If in addition $k_g\ge0$ when there is a boundary and the covariance is nonpositive, then $H_+$ is nonincreasing. A sufficient condition is $r_+=p_t(S)$ with $p_t$ nonincreasing. If $\phi$ and $S$ are both nonconstant at a given time, the decrease is strict at that time.
\end{proposition}
\begin{proof}
The zero-flux condition gives $A_+'=\int_MQ\dv$. Differentiate $\frac{E_\partial}{A_+}$ and use \eqref{eq:square}; the two source terms combine into the covariance in \eqref{eq:positivecov}. The sign follows from \eqref{eq:probabilitycov}.
The strict assertion follows from Lemma~\ref{lem:rigidity}.
\end{proof}
At a fixed time \eqref{eq:positivecov} also gives the exact criterion $H_+^{\prime}\le0$ if and only if $A_+\operatorname{Cov}_{\nu_+}(r_+,\log S)\le\mathcal D$, independently of the sign of $k_g$. For example, $Q=c(t)S$ makes the covariance vanish, although $Q$ increases with $S$ when $c>0$. This criterion concerns $H_+$, not $E_\partial$: in this case $E_\partial'=-\mathcal D+cE_\partial$.

\section{Free-boundary harmonic disks}\label{sec:free}
Fix a smooth embedded submanifold $P\subset N$. The free-boundary conditions are
\begin{equation}\label{eq:free}
\phi(\partial M)\subset P,\qquad \phi_n\perp T_\phi P.
\end{equation}
They allow $\phi_n\ne0$ and are distinct from a zero Neumann condition on the map. For background on harmonic maps and free-boundary analysis see \cite{HeleinWood,JostLiuZhu}.

\begin{lemma}[Weak conformality on a disk]\label{lem:hopf}
A smooth harmonic map from a compact smooth topological disk satisfying \eqref{eq:free} is weakly conformal:
\begin{equation}\label{eq:conformal}
\phi^*h=\tfrac12|\nabla\phi|^2g.
\end{equation}
\end{lemma}
\begin{proof}
We recall the classical Hopf-differential argument. Use conformal disk coordinates and the holomorphic Hopf coefficient
$q(z)=\langle\phi_z,\phi_z\rangle_h$, where the metric pairing is extended complex bilinearly; see \cite{HeleinWood}. The boundary orthogonality gives $\langle\phi_r,\phi_\theta\rangle_h=0$, so $\operatorname{Im}(z^2q)=0$ on the unit circle. The maximum principle and $(z^2q)(0)=0$ give $z^2q\equiv0$, hence $q\equiv0$. This is \eqref{eq:conformal}, including possible branch points.
\end{proof}

\begin{lemma}[Normal energy derivative]\label{lem:normalfree}
Suppose $\phi$ is harmonic and satisfies \eqref{eq:free}. Put
$V_\phi=\dd\phi(e)$, $W_\phi=\phi_n$, and
$\mathrm{II}_P(U,V)=(\nabla^N_UV)^{\perp_P}$. Then
\begin{align}
\partial_n|\nabla\phi|^2
&=-4\langle\mathrm{II}_P(V_\phi,V_\phi),W_\phi\rangle_h-2k_g|\nabla\phi|^2,\label{eq:normalfree}\\
S_n&=R_n+4\alpha\langle\mathrm{II}_P(V_\phi,V_\phi),W_\phi\rangle_h
+2\alpha k_g|\nabla\phi|^2.\label{eq:Snfree}
\end{align}
\end{lemma}
\begin{proof}
Extend $(e,n)$ by normal parallel transport and set $H=\nabla\dd\phi$. Along the boundary,
\[
H(e,e)=\nabla^\phi_eV_\phi+k_gW_\phi,
\quad H(e,n)=\nabla^\phi_eW_\phi-k_gV_\phi,
\quad H(n,n)=-H(e,e).
\]
It follows that
\begin{align*}
\partial_n|\nabla\phi|^2
={}&2\langle\nabla^\phi_eW_\phi,V_\phi\rangle_h
-2\langle\nabla^\phi_eV_\phi,W_\phi\rangle_h
-2k_g(|V_\phi|^2+|W_\phi|^2)\\
={}&-4\langle\mathrm{II}_P(V_\phi,V_\phi),W_\phi\rangle_h-2k_g|\nabla\phi|^2.
\end{align*}
The second equality uses the tangential derivative of $\langle V_\phi,W_\phi\rangle_h=0$ and $W_\phi\perp TP$. Equation~\eqref{eq:Snfree} follows from the definition of $S$.
\end{proof}
Even for totally geodesic $P$, the conditions $S_n=0$ and $R_n=0$ need not be equivalent. They agree if also $k_g=0$.

\begin{theorem}[Free-boundary disk monotonicity]\label{thm:free}
Let $(g(t),\phi)$ be a smooth harmonic Ricci flow with a fixed harmonic map on a compact topological disk. Assume \eqref{eq:free} and $S_n=0$ at every time. If $S(\cdot,0)>0$ and $k_g(\cdot,0)\ge0$, then these inequalities persist and
\begin{equation}\label{eq:freedecay}
E_\partial'=-\mathcal D\le0.
\end{equation}
If $\phi$ is nonconstant, equality at a given time holds if and only if $S$ is spatially constant. Branch points are allowed.
\end{theorem}
\begin{proof}
Lemma~\ref{lem:hopf} and harmonicity give $Q=0$ and $\Sc=\frac{S}{2}g$, so
\begin{equation}\label{eq:conformalflow}
g_t=-Sg,\qquad S_t=\Delta S+S^2.
\end{equation}
The Neumann maximum principle preserves positivity. The conformal variation of boundary curvature gives
$k_g'=\frac{Sk_g}{2}-\frac{S_n}{2}=\frac{Sk_g}{2}$, and thus
\[
k_g(x,t)=k_g(x,0)\exp\left(\tfrac12\int_0^tS(x,s)\,\dd s\right)\ge0.
\]
Equation~\eqref{eq:square} proves \eqref{eq:freedecay}; Lemma~\ref{lem:rigidity} gives the equality statement.
\end{proof}
For a constant map this recovers the Ricci-flow entropy formula of \cite[Theorem 1.4]{MR4785569}. The theorem is conditional on smooth existence of the coupled flow with the stated boundary conditions.

\begin{example}[Strict decay on a hemisphere]\label{ex:hemisphere}
Let $h$ be the unit round metric, let $M=S^2_+$, $N=S^2$, and let $P$ be the equator. Take the inclusion map and
\[
0<\alpha<1,\qquad g(t)=e^{2\omega(t)}h,
\qquad\omega(\cdot,0)=\varepsilon z^2,
\qquad 0<\varepsilon<\min\{\frac{1-\alpha}{4},\frac{1}{8}\},
\]
where $z$ is the height function. There is a short-time smooth solution satisfying Theorem~\ref{thm:free}, with $E_\partial'(0)<0$.
\end{example}
\begin{proof}
Two-dimensional conformal invariance gives $\tau(\phi)=0$ and $|\nabla\phi|^2=2e^{-2\omega}$. The conformal curvature formula gives
\[
S=2e^{-2\omega}(1-\alpha-\Delta_h\omega),
\qquad \omega_t=e^{-2\omega}(\Delta_h\omega+\alpha-1).
\]
Solve this locally parabolic equation first on the entire sphere, using \cite[Theorem 1.1]{MantegazzaMartinazzi}. Uniqueness preserves equatorial reflection symmetry, so restriction to the hemisphere has $k_g=S_n=0$ and satisfies the genuine free-boundary condition. Since $\Delta_hz^2=2-6z^2$,
\[
S_0=2e^{-2\varepsilon z^2}(1-\alpha-2\varepsilon+6\varepsilon z^2)>0.
\]
As a function of $s=z^2$, its derivative is
$4\varepsilon e^{-2\varepsilon s}(2+\alpha+2\varepsilon-6\varepsilon s)>0$ on $[0,1]$. Thus $S_0$ is nonconstant, and strict decay follows from $|\nabla\phi|^2>0$ and Theorem~\ref{thm:free}.
\end{proof}

\section{A spectral criterion}\label{sec:spectral}
Here $M$ is any compact connected surface. If it has boundary, assume $S_n=0$ but impose no sign on $k_g$. Let $\lambda_1(g)>0$ be the first nonzero Neumann eigenvalue of $-\Delta_g$, or the ordinary first nonzero eigenvalue in the closed case. All norms and averages below use $g(t)$.

\begin{lemma}[Logarithmic variance]\label{lem:logvariance}
If $S\ge a>0$, then
\begin{equation}\label{eq:logvariance}
\|\log S-\overline{\log S}\|_{L^2}
\le a^{-1}\|S-\bar S\|_{L^2}.
\end{equation}
If also $S\le b$, inequality~\eqref{eq:Jensen} holds.
\end{lemma}
\begin{proof}
Orthogonal projection onto constants minimizes the $L^2$ distance. Since $\log$ is $a^{-1}$-Lipschitz on $[a,\infty)$,
\[
\|\log S-\overline{\log S}\|_2
\le\|\log S-\log\bar S\|_2\le a^{-1}\|S-\bar S\|_2.
\]
For the second assertion, integrate
\[
-\log S\ge-\log\bar S-\frac{S-\bar S}{\bar S}
+\frac{(S-\bar S)^2}{2b^2}.
\]
\end{proof}

\begin{theorem}[Spectral gap, source oscillation and mean source]\label{thm:spectral}
Let $(g(t),\phi(t))$ be a smooth solution of \eqref{RH} with $S>0$ and $S_n=0$ if there is a boundary. Suppose at each time that
\begin{equation}\label{eq:spectralhyp}
0<a(t)\le S\le b(t),\qquad
\|Q-\bar Q\|_2\le L(t)\|S-\bar S\|_2,\quad L(t)\ge0.
\end{equation}
Define
\begin{equation}\label{eq:eta}
\eta(t)=\frac{\lambda_1(g(t))}{b(t)}-1-\frac{L(t)}{a(t)}
+\frac{\bar Q(t)}{2b(t)^2}.
\end{equation}
Then
\begin{equation}\label{eq:spectralresult}
E_\partial'(t)\le-\eta(t)\int_M(S-\bar S)^2\dv.
\end{equation}
In particular, $\eta\ge0$ throughout the interval implies monotonicity. If $\eta(t)>0$ and $S(\cdot,t)$ is nonconstant, then $E_\partial'(t)<0$.
\end{theorem}
\begin{proof}
Writing $D=\int_M(S-\bar S)^2\dv$, the Poincar\'e inequality gives
\[
\int_M\frac{|\nabla S|^2}{S}\dv\ge\frac1b\int_M|\nabla S|^2\dv
\ge\frac{\lambda_1}bD.
\]
Decompose the source integral exactly as
\[
\int_MQ\log\frac S{\bar S}\dv
=\int_M(Q-\bar Q)(\log S-\overline{\log S})\dv
+\bar Q\int_M\log\frac S{\bar S}\dv.
\]
By Cauchy--Schwarz, \eqref{eq:spectralhyp}, \eqref{eq:logvariance} and \eqref{eq:Jensen},
\[
\int_MQ\log\frac S{\bar S}\dv
\le\left(\frac La-\frac{\bar Q}{2b^2}\right)D.
\]
Insert these bounds into \eqref{eq:balance}.
\end{proof}

\begin{corollary}[Quantitative decay]\label{cor:spectraldecay}
If $\eta\ge0$ and $\eta\bar S$ is locally integrable, then
\begin{equation}\label{eq:positivedecay}
E_\partial(t)\le E_\partial(0)
\exp\left(-\int_0^t\eta(s)\bar S(s)\,\dd s\right).
\end{equation}
\end{corollary}
\begin{proof}
The inequality $\log x\le x-1$ gives $E_\partial\le \frac{D}{\bar S}$. Thus \eqref{eq:spectralresult} implies $E_\partial'\le-\eta\bar S E_\partial$, and integration proves the claim.
\end{proof}
\begin{remark}[Comparison and scope]
Discarding the mean-source contribution yields only the coefficient $\frac{\lambda_1}{b}-1-\frac{L}{a}$. The additional nonnegative term $\frac{\bar Q}{2b^2}$ therefore strengthens that estimate without changing its hypotheses. No optimality of this coefficient is claimed. The oscillation and spectral assumptions are conditions along the solution; their preservation from initial data is not proved. For a fixed harmonic map,
$Q=2\alpha^2|\T-\frac{|\nabla\phi|^2}{2}g|^2$, so the oscillation condition controls the spatial variation of the conformality defect. At times when $S$ is constant, it forces $Q$ to be constant as well. No non-disk boundary example satisfying these spectral hypotheses is constructed here.
\end{remark}

\section{Scope of the results}
Equations~\eqref{eq:thresholdH} and \eqref{eq:thresholdE} determine the sign at a fixed time. The variance estimate \eqref{eq:oscweak} and the ordering of the source ratio give different sufficient tests. The closed examples in Section~\ref{sec:balances} show why an unconditional decay theorem fails. The torus example realizes a pointwise ordering, and its perturbations realize strict covariance positivity with a nonzero tension field. Their construction proves these properties only for a short time. A geometric condition guaranteeing the covariance sign throughout the smooth lifespan remains to be found. In the proportional-source subclass, the mass and area are explicit, but the examples with $Q>0$ above do not have spatially constant $\frac{Q}{\rho}$.

The positive-curvature square identity extends the Ricci-flow argument by retaining the coupling and source terms. The sufficient conditions proved here use either boundary convexity and control of the source, or the spectral assumptions in Section~\ref{sec:spectral}. Neither criterion establishes long-time existence. Likewise, the free-boundary formulas assume a smooth compatible flow; they do not constitute a general existence theorem for the coupled boundary problem.

\clearpage
\appendix
\section{The Perelman-type \texorpdfstring{$\mathcal F$}{F}-functional}\label{sec:F}
For variations we retain the notation
\[
\delta g_{ij}=v_{ij},\quad v=g^{ij}v_{ij},\quad
v^{ij}=g^{ik}g^{j\ell}v_{k\ell},\quad
\delta f=h,\quad\delta\phi=\theta,\quad\delta\tau=\sigma.
\]
The scalar variation $h$ is distinguished by context from the target metric $h$. The expression $v(n,\nabla f)$ uses the variation tensor, while $v$ without arguments is its trace. In the main text, $\psi$ denotes a Poisson auxiliary function and $\omega$ a conformal factor; neither is identified with the Perelman weight $f$ or density $u$.

Following \cite[Section 3.1]{MR2961788}, set
\begin{equation}\label{FV}
\mathcal F(g,\phi,f)=\int_M(S+|\nabla f|^2)e^{-f}\dv.
\end{equation}
For a constant map this is denoted by $F(g,f)$.

\begin{proposition}[First variation]\label{thm:Fvariation}
For an arbitrary smooth variation,
\begin{equation}\label{eq:Fvariation}
\begin{aligned}
\delta\mathcal F={}&\int_M\left[-v^{ij}(\Sc_{ij}+\nabla_i\nabla_jf)
+\left(\frac v2-h\right)(2\Delta f-|\nabla f|^2+S)\right]e^{-f}\dv\\
&+2\alpha\int_M\langle\theta,\tau(\phi)-\dd\phi(\nabla f)\rangle_h e^{-f}\dv
+\mathcal B(v,h,\theta),
\end{aligned}
\end{equation}
where
\begin{equation}\label{eq:B}
\begin{aligned}
\mathcal B(v,h,\theta)=\int_{\partial M}\big[&
(\nabla^iv_{ij})n^j-\partial_nv+v(n,\nabla f)
-(v-2h)f_n\\
&-2\alpha\langle\phi_n,\theta\rangle_h\big]e^{-f}\dA.
\end{aligned}
\end{equation}
\end{proposition}
\begin{proof}
For the Ricci part, use the boundary variation formula in \cite[Proposition 3.1]{MR4003012}. For the map part, the standard identities are
\[
\delta|\nabla\phi|^2=-v^{ij}\T_{ij}+2\langle\nabla\theta,\nabla\phi\rangle_h,
\qquad \delta(e^{-f}\dd v_g)=(\frac{v}{2}-h)e^{-f}\dd v_g.
\]
Weighted integration by parts yields
\begin{align*}
\delta\int_M(-\alpha|\nabla\phi|^2)e^{-f}\dv
={}&\int_M\left[\alpha v^{ij}\T_{ij}-\alpha|\nabla\phi|^2(\frac{v}{2}-h)\right]e^{-f}\dv\\
&+2\alpha\int_M\langle\theta,\tau(\phi)-\dd\phi(\nabla f)\rangle_he^{-f}\dv
-2\alpha\int_{\partial M}\langle\phi_n,\theta\rangle_he^{-f}\dA.
\end{align*}
Adding the two variations proves the claim; compare \cite[Section 3.1]{MR2961788} for the closed case.
\end{proof}

Consider the original system
\begin{equation}\label{q-3}
\begin{cases}
g_t=-2\Sc,\quad\phi_t=\tau(\phi),&\text{on }M,\\
f_t=-\Delta f+|\nabla f|^2-S,&\text{on }M,\\
f_n=0,&\text{on }\partial M.
\end{cases}
\end{equation}
The density $e^{-f}$ satisfies the conjugate heat equation. The vector field $-\nabla f$ is tangent to the boundary and generates boundary-preserving diffeomorphisms. Pulling back gives
\begin{equation}\label{q-1}
\begin{cases}
g_t=-2(\Sc+\nabla^2f),\\
\phi_t=\tau(\phi)-\dd\phi(\nabla f),\\
f_t=-\Delta f-S,\qquad f_n=0.
\end{cases}
\end{equation}
Prescribed boundary data are pulled back with the solution: if $k_g(x,t)=\chi(x,t)$, the new data are $\chi(\Phi_t(x),t)$. In particular, a time derivative of prescribed boundary data cannot be treated as a pullback-invariant scalar without accounting for transport.

\begin{theorem}[Evolution of $\mathcal F$]\label{thm:Fevolution}
Along a smooth solution of \eqref{q-3},
\begin{equation}\label{eq:Fevolution}
\begin{aligned}
\mathcal F'={}&2\int_M|\Sc+\nabla^2f|^2e^{-f}\dv
+2\alpha\int_M|\tau(\phi)-\dd\phi(\nabla f)|^2e^{-f}\dv\\
&+\int_{\partial M}\left[S_n
+4\alpha\langle\phi_n,\dd\phi(\nabla^\top f)\rangle_h
+2k_g|\nabla^\top f|^2\right]e^{-f}\dA.
\end{aligned}
\end{equation}
If $\phi_n=0$, the mixed map term in the boundary integral vanishes.
\end{theorem}
\begin{proof}
By diffeomorphism invariance, compute in \eqref{q-1}. Substitute
$v_{ij}=-2(\Sc_{ij}+\nabla_i\nabla_jf)$, $h=-S-\Delta f$ and
$\theta=\tau(\phi)-\dd\phi(\nabla f)$ in \eqref{eq:Fvariation}. Since $\frac{v}{2}-h=0$, the interior terms are the stated squares. For the boundary integrand, use
\begin{equation}\label{eq:divergence}
\nabla^i\Sc_{ij}=\tfrac12\nabla_jS-\alpha\langle\tau(\phi),\nabla_j\phi\rangle_h,
\qquad \diver\nabla^2f=\dd\Delta f+\Ric(\nabla f,\cdot).
\end{equation}
After cancellation of the tension and $\partial_n\Delta f$ terms, it becomes
$S_n-4\Sc(n,\nabla f)-2\nabla^2f(n,\nabla f)$.
The identities $f_n=0$ and $\Ric=\frac{R}{2}g$ imply
\[
\nabla^2f(n,\nabla f)=-k_g|\nabla^\top f|^2,
\quad \Sc(n,\nabla f)=-\alpha\langle\phi_n,\dd\phi(\nabla^\top f)\rangle_h.
\]
This proves \eqref{eq:Fevolution}.
\end{proof}

\begin{lemma}[Boundary curvature evolution]\label{lem:boundary}
Suppose $\phi_n=0$ along the boundary throughout a smooth solution of \eqref{RH}. Then
\begin{align}
\partial_n|\nabla\phi|^2&=-2k_g|\nabla^\top\phi|^2,\label{eq:normalNeumann}\\
R_n&=k_gR-4\alpha k_g|\nabla^\top\phi|^2-2k_g',\label{eq:Rnormal}\\
S_n&=k_g(R-2\alpha|\nabla^\top\phi|^2)-2k_g'
=2k_g\Sc(e,e)-2k_g'.\label{eq:Snormal}
\end{align}
\end{lemma}
\begin{proof}
Tangential differentiation of $\phi_n=0$ gives
$(\nabla\dd\phi)(n,e)=-k_g\dd\phi(e)$, proving \eqref{eq:normalNeumann}.
For a metric variation $v=g_t$, the standard connection variation gives
\[
k_g'=\tfrac12(\nabla_nv)(e,e)-(\nabla_ev)(e,n)
+\bigl(\tfrac12v(n,n)-v(e,e)\bigr)k_g.
\]
Here $v=-Rg+2\alpha\T$, $v(e,n)=0$, $v(n,n)=-R$, and
$(\nabla_ev)(e,n)=k_g[v(n,n)-v(e,e)]$. Consequently
\[
k_g'=-\tfrac12R_n-2\alpha k_g|\nabla^\top\phi|^2+\tfrac12k_gR.
\]
This proves \eqref{eq:Rnormal}; subtracting $\alpha\partial_n|\nabla\phi|^2$ gives \eqref{eq:Snormal}. When the map is constant, the formula reduces to \cite[Proposition 2.1]{MR4003012}.
\end{proof}

\begin{corollary}[Monotonicity of $\mathcal F$]\label{thm:Fmono}
If \eqref{q-3} holds and
\begin{equation}\label{eq:boundarysign}
\phi_n=0,\qquad k_g\ge0,\qquad S_n\ge0
\end{equation}
throughout the time interval, then $\mathcal F$ is nondecreasing. A sufficient set of conditions is
\begin{equation}\label{eq:boundarysufficient}
\phi_n=0,\qquad k_g\ge0,\qquad k_g'\le0,\qquad
R-2\alpha|\nabla^\top\phi|^2\ge0\quad\text{on }\partial M.
\end{equation}
The same assertion holds for the corresponding pullback solution.
\end{corollary}
\begin{proof}
Apply \eqref{eq:Fevolution} and \eqref{eq:Snormal}.
\end{proof}
\begin{remark}
The conditions $k_g\ge0$ and $k_g'\le0$ alone do not determine the sign of $S_n$. Even for a constant map, the term $k_gR$ requires control. If the map is constant on each boundary component, its tangential derivative vanishes, but this does not replace $\phi_n=0$. If $k_g\equiv0$ throughout the interval and $\phi_n=0$, then $k_g'=S_n=0$, and no additional curvature sign is needed for the boundary contribution.
\end{remark}

\section{The Perelman-type \texorpdfstring{$\mathcal W$}{W}-functional}\label{sec:W}
Set
\begin{equation}\label{eq:Wdef}
u=(4\pi\tau)^{-1}e^{-f},\qquad
\mathcal W_{RH}(g,f,\phi,\tau)=\int_M[\tau(S+|\nabla f|^2)+f-2]u\dv.
\end{equation}
This is the surface version of the coupled functional in \cite[Section 7.1]{MR2961788}. For a constant map it is Perelman's functional \cite{zbMATH05050939}. Normalizing $\int_Mu\dv=1$ is optional for the identities below.

\begin{proposition}[First variation of $\mathcal W_{RH}$]\label{thm:Wvariation}
If $\delta\tau=\sigma$ is spatially constant, then
\begin{equation}\label{eq:Wvariation}
\begin{aligned}
\delta\mathcal W_{RH}={}&\int_M(-\tau v^{ij}+\sigma g^{ij})
\left(\Sc_{ij}+\nabla_i\nabla_jf-\frac{g_{ij}}{2\tau}\right)u\dv\\
&+2\tau\alpha\int_M\langle\theta,\tau(\phi)-\dd\phi(\nabla f)\rangle_hu\dv\\
&+\int_M\left(\frac v2-h-\frac\sigma\tau\right)
[\tau(2\Delta f-|\nabla f|^2+S)+f-3]u\dv+\mathcal A,
\end{aligned}
\end{equation}
where
\begin{equation}\label{eq:Aboundary}
\mathcal A=\frac1{4\pi}\mathcal B(v,h,\theta)
+\sigma\int_{\partial M}f_nu\dA.
\end{equation}
No Neumann condition is imposed in this variation formula.
\end{proposition}
\begin{proof}
Use $\mathcal W_{RH}=(4\pi)^{-1}\mathcal F+\int_M(f-2)u\dv$ and
$\delta(u\dd v_g)=(\frac{v}{2}-h-\frac{\sigma}{\tau})u\dd v_g$. Substitution of \eqref{eq:Fvariation} and rearrangement give the displayed interior terms; the remaining difference is
\[
\sigma\int_M(\Delta f-|\nabla f|^2)u\dv
=\sigma\int_{\partial M}f_nu\dA.
\]
This gives \eqref{eq:Aboundary}.
\end{proof}

Consider
\begin{equation}\label{p-2}
\begin{cases}
g_t=-2\Sc,\quad\phi_t=\tau(\phi),\\
f_t=-\Delta f+|\nabla f|^2-S+\frac{1}{\tau},\quad\tau_t=-1,\\
f_n=0,\quad\phi_n=0\quad\text{on }\partial M.
\end{cases}
\end{equation}
Then $(-\partial_t-\Delta+S)u=0$, $u_n=0$, and $\int_Mu\dv$ is constant. Pullback by $-\nabla f$ replaces the first two equations by $g_t=-2(\Sc+\nabla^2f)$ and $\phi_t=\tau(\phi)-\dd\phi(\nabla f)$, and the weight equation by $f_t=-\Delta f-S+\frac{1}{\tau}$. No independent condition $R_n=0$ is imposed; the relevant relation is \eqref{eq:Rnormal}.

\begin{theorem}[Evolution and monotonicity of $\mathcal W_{RH}$]\label{thm:Wmono}
For a smooth solution of \eqref{p-2} with $\tau>0$,
\begin{equation}\label{eq:Wevolution}
\begin{aligned}
\mathcal W_{RH}'={}&2\tau\int_M\left|\Sc+\nabla^2f-\frac g{2\tau}\right|^2u\dv
+2\tau\alpha\int_M|\tau(\phi)-\dd\phi(\nabla f)|^2u\dv\\
&+\frac1{4\pi}\int_{\partial M}(S_n+2k_g|\nabla^\top f|^2)e^{-f}\dA.
\end{aligned}
\end{equation}
Thus $k_g\ge0$ and $S_n\ge0$ imply that $\mathcal W_{RH}$ is nondecreasing. Conditions \eqref{eq:boundarysufficient} are sufficient, and the result is invariant under the corresponding pullback.
\end{theorem}
\begin{proof}
Put $\widehat f=f+\log(4\pi\tau)$, so $(g,\phi,\widehat f)$ satisfies \eqref{q-3}. Write
\[
E=\int_M(S+|\nabla f|^2)u\dv,\qquad
H=\int_Mfu\dv,\qquad m_0=\int_Mu\dv.
\]
The conjugate heat equation and the Neumann conditions give $m_0'=0$ and $H'=-E+\frac{m_0}{\tau}$. Hence $\mathcal W_{RH}'=\tau E'-2E+\frac{m_0}{\tau}$.
Apply \eqref{eq:Fevolution} to $\widehat f$ and complete the square, using
$\int_M(S+\Delta f)u\dv=E$ and $|g|^2=2$. This gives \eqref{eq:Wevolution}; the monotonicity statements follow from \eqref{eq:Snormal}.
\end{proof}

\begin{corollary}[Boundary-constant maps]\label{cor:boundaryconstant}
Under \eqref{p-2}, if $\phi$ is spatially constant on each boundary component, the boundary integrand in \eqref{eq:Wevolution} is
\[
k_gR-2k_g'+2k_g|\nabla^\top f|^2.
\]
Thus $k_g\ge0$, $k_g'\le0$ and $R|_{\partial M}\ge0$ suffice for monotonicity.
\end{corollary}
\begin{proof}
Use $\nabla^\top\phi=0$ in \eqref{eq:Snormal}.
\end{proof}
For a globally constant map, these formulas recover the Ricci-flow case in \cite[Theorem 3.2]{MR4003012}. For closed surfaces all boundary terms disappear, recovering the corresponding two-dimensional formulas of \cite[Propositions 3.1 and 7.1]{MR2961788}.

\begin{corollary}[Orthogonal map data]\label{cor:orthogonal}
In Theorem~\ref{thm:Wmono}, replace $\phi_n=0$ by
\begin{equation}\label{eq:orthogonal}
\langle\phi_n,\dd\phi(e)\rangle_h=0\quad\text{on }\partial M.
\end{equation}
Then \eqref{eq:Wevolution} still holds, and $k_g\ge0$, $S_n\ge0$ imply monotonicity. The same sign conditions imply monotonicity of $\mathcal F$ in Theorem~\ref{thm:Fevolution}. In particular, these assertions apply to smooth flows with $\phi(\partial M,t)\subset P$ and $\phi_n\perp T_\phi P$, where $P\subset N$ is a fixed smooth submanifold. No harmonicity assumption on $\phi(t)$ is needed.
\end{corollary}
\begin{proof}
Since $\nabla f$ is tangential on the boundary, \eqref{eq:orthogonal} eliminates the mixed map term in \eqref{eq:Fevolution}. The proof of Theorem~\ref{thm:Wmono} then applies unchanged. For free-boundary data, $\dd\phi(e)\in T_\phi P$, which gives \eqref{eq:orthogonal}.
\end{proof}
The boundary curvature identities in Lemma~\ref{lem:boundary} still require $\phi_n=0$; the orthogonality condition alone does not justify them.

\section*{Use of generative AI}
This paper originate from the work \cite{CaoEarlier} which is done without any AI tools, Six month later, I used Chatgpt to improve it. For this paper,   ChatGPT (OpenAI) assisted with drafting, mathematical argument development,
literature checks, and preparation of the accompanying Lean calculations.
The author is responsible for the manuscript.


\begin{thebibliography}{99}
\small
\raggedright
\bibitem{Hamilton1988}
Hamilton R S. The Ricci flow on surfaces.
In: Mathematics and general relativity (Santa Cruz, CA, 1986).
Contemporary Mathematics, vol. 71. American Mathematical Society, 1988: 237--262.
\url{https://doi.org/10.1090/conm/071/954419}.

\bibitem{zbMATH05050939}
Perelman G. The entropy formula for the Ricci flow and its geometric applications.
arXiv:math/0211159, 2002.
\url{https://arxiv.org/abs/math/0211159}.

\bibitem{List2008}
List B. Evolution of an extended Ricci flow system.
Communications in Analysis and Geometry, 2008, 16(5): 1007--1048.
\url{https://doi.org/10.4310/CAG.2008.v16.n5.a5}.

\bibitem{MR2961788}
M\"uller R. Ricci flow coupled with harmonic map flow.
Annales scientifiques de l'\'Ecole Normale Sup\'erieure,
2012, 45(1): 101--142.
\url{https://doi.org/10.24033/asens.2161}.

\bibitem{MR4785569}
Kunikawa K, Sakurai Y. Hamilton-type entropy formula along the Ricci flow
on surfaces with boundary.
Communications in Analysis and Geometry, 2023, 31(7): 1655--1668.
\url{https://doi.org/10.4310/cag.2023.v31.n7.a2}.

\bibitem{MR3300708}
Williams M B. Results on coupled Ricci and harmonic map flows.
Advances in Geometry, 2015, 15(1): 7--26.
\url{https://doi.org/10.1515/advgeom-2014-0026}.

\bibitem{MR3729736}
Sun J. Conditions to extend the Ricci flow coupled with harmonic map flow.
International Journal of Mathematics, 2017, 28(12): 1750091, 31 pp.
\url{https://doi.org/10.1142/S0129167X17500914}.

\bibitem{MR3175257}
Guo H, He T. Harnack estimates for geometric flows, applications to
Ricci flow coupled with harmonic map flow.
Geometriae Dedicata, 2014, 169: 411--418.
\url{https://doi.org/10.1007/s10711-013-9864-z}.

\bibitem{MR3163480}
Li Y. Eigenvalues and entropies under the harmonic-Ricci flow.
Pacific Journal of Mathematics, 2014, 267(1): 141--184.
\url{https://doi.org/10.2140/pjm.2014.267.141}.

\bibitem{MR3547931}
Li Y. Long time existence of Ricci-harmonic flow.
Frontiers of Mathematics in China, 2016, 11(5): 1313--1334.
\url{https://doi.org/10.1007/s11464-016-0579-y}.

\bibitem{MR3037998}
Guo H. An entropy formula relating Hamilton's surface entropy and
Perelman's $\mathcal W$ entropy.
Comptes Rendus Math\'ematique, 2013, 351(3--4): 115--118.
\url{https://doi.org/10.1016/j.crma.2013.03.003}.

\bibitem{MR4003012}
Cortissoz J C, Murcia A. The Ricci flow on surfaces with boundary.
Communications in Analysis and Geometry, 2019, 27(2): 377--420.
\url{https://doi.org/10.4310/CAG.2019.v27.n2.a5}.

\bibitem{MR4452202}
Bracken P. Monotonicity properties of functionals under Ricci flow
on manifolds without and with boundary.
International Journal of Geometric Methods in Modern Physics,
2022, 19(9): 2250137, 21 pp.
\url{https://doi.org/10.1142/S0219887822501377}.

\bibitem{MR3557306}
Gianniotis P. The Ricci flow on manifolds with boundary.
Journal of Differential Geometry, 2016, 104(2): 291--324.
\url{https://doi.org/10.4310/jdg/1476367059}.

\bibitem{MR3448425}
Gianniotis P. Boundary estimates for the Ricci flow.
Calculus of Variations and Partial Differential Equations,
2016, 55(1): Article 9, 21 pp.
\url{https://doi.org/10.1007/s00526-015-0938-7}.

\bibitem{MR4959090}
Jouttij\"arvi R. Novel boundary conditions for the Ricci flow.
The Journal of Geometric Analysis, 2025, 35(11): Article 360, 43 pp.
\url{https://doi.org/10.1007/s12220-025-02181-3}.

\bibitem{MR1387799}
Shen Y. On Ricci deformation of a Riemannian metric on manifold with boundary.
Pacific Journal of Mathematics, 1996, 173(1): 203--221.
\url{https://doi.org/10.2140/pjm.1996.173.203}.

\bibitem{CaoEarlier}
Cao X.-Z. Notes on harmonic-Ricci flow on surface.
arXiv:2605.05599v1, 2026.
\url{https://arxiv.org/abs/2605.05599v1}.

\bibitem{MantegazzaMartinazzi}
Mantegazza C, Martinazzi L. A note on quasilinear parabolic equations on manifolds.
Annali della Scuola Normale Superiore di Pisa, Classe di Scienze (5), 2012, 11(4): 857--874.
\url{https://www.numdam.org/item/ASNSP_2012_5_11_4_857_0/}.

\bibitem{Reilly}
Reilly R C. Applications of the Hessian operator in a Riemannian manifold.
Indiana University Mathematics Journal, 1977, 26(3): 459--472.
\url{https://doi.org/10.1512/iumj.1977.26.26036}.

\bibitem{HeleinWood}
H\'elein F, Wood J C. Harmonic maps.
In: Krupka D, Saunders D, eds. Handbook of Global Analysis.
Elsevier, 2008: 417--491.
\url{https://doi.org/10.1016/B978-044452833-9.50009-7}.

\bibitem{JostLiuZhu}
Jost J, Liu L, Zhu M. The qualitative behavior at the free boundary
for approximate harmonic maps from surfaces.
Mathematische Annalen, 2019, 374: 133--177.
\url{https://doi.org/10.1007/s00208-018-1759-8}.

\end{thebibliography}
\end{document}